\documentclass{amsart} 
\usepackage{amsmath,amssymb,amsthm}

\makeatletter           
\renewcommand\normalsize{
   \@setfontsize\normalsize\@xpt\@xiipt
   \abovedisplayskip 5\p@ \@plus2\p@ \@minus5\p@   
   \abovedisplayshortskip \z@ \@plus3\p@
   \belowdisplayshortskip 6\p@ \@plus3\p@ \@minus3\p@
   \belowdisplayskip \abovedisplayskip
   \let\@listi\@listI}
\makeatother

\usepackage{graphicx}
\allowdisplaybreaks
\usepackage{enumerate}  
\usepackage{indentfirst}    
\theoremstyle{plain}
\newtheorem{theorem}{Theorem}[section]
\newtheorem{proposition}[theorem]{Proposition}
\newtheorem{lemma}[theorem]{Lemma}
\newtheorem{corollary}[theorem]{Corollary}

\theoremstyle{definition}
\newtheorem{definition}[theorem]{Definition}    

\theoremstyle{remark}
\newtheorem{remark}[theorem]{Remark}

\begin{document}
\title[Left derivable or Jordan left derivable mappings]{Left derivable or Jordan left derivable mappings on Banach algebras}
\author[Y. Ding]{Yana Ding}
\address[Yana Ding]{Department of Mathematics, East China University of Science and Technology, Shanghai, P. R. China}
\email{dingyana@mail.ecust.edu.cn}
\author[J. Li]{Jiankui Li $^\ast$$^\dag$}
\address[Jiankui Li]{Department of Mathematics, East China University of Science and Technology, Shanghai, P. R. China}
\email{jiankuili@yahoo.com}
\thanks{$^\ast$Corresponding author. E-mail address: jiankuili@yahoo.com}
\thanks{$^\dag$This research was partially supported by National Natural Science Foundation of China (Grant No. 11371136).}
\begin{abstract}
Let $\delta$ be a linear mapping from a unital Banach algebra $\mathcal{A}$ into a unital left $\mathcal{A}$-module $\mathcal{M}$, and $W$ in $\mathcal{Z}(\mathcal{A})$ be a left separating point of $\mathcal{M}$. We show that the following three conditions are equivalent: ($i$) $\delta$ is a Jordan left derivation; ($ii$)$\delta$ is left derivable at $W$; ($iii$) $\delta$ is Jordan left derivable at $W$. Let $\mathcal{A}$ be a Banach algebra with the property ($\mathbb{B}$) (see Definition \ref{Prop_B}), and $\mathcal{M}$ be a Banach left $\mathcal{A}$-module. We consider the relations between generalized (Jordan) left derivations and (Jordan) left derivable mappings at zero from $\mathcal{A}$ into $\mathcal{M}$.\\
\textbf{Keywords:}  (Jordan) left derivation, generalized (Jordan) left derivation, (Jordan) left derivable mapping.  \\
\textbf{MSC(2010):}  47B47, 47B49, 47L10.
\end{abstract}
\maketitle

\vspace{1em}
\section{Introduction}

Let $\mathcal{A}$ be a Banach algebra over the complex field $\mathbb{C}$. As usual, for each $A,B$ in $\mathcal{A}$, we write $A\circ B$ for the \emph{Jordan product} $AB+BA$, while $A\cdot B$ or $AB$ denotes the ordinary \emph{product} of $A$ and $B$. The \emph{centre} of $\mathcal{A}$ is $\mathcal{Z}(\mathcal{A})=\{A\in\mathcal{A}:AB=BA$ for each $B$ in $\mathcal{A}\}$. Let $\mathcal{M}$ be a left $\mathcal{A}$-module. The \emph{right annihilator} of $\mathcal{A}$ on $\mathcal{M}$ is $Rann_{\mathcal{M}}(\mathcal{A})=\{M\in\mathcal{M}:AM=0$ for each $A$ in $\mathcal{A}\}$. An element $W$ in $\mathcal{A}$ is said to be \emph{a left separating point of $\mathcal{M}$}, if for each $M$ in $\mathcal{M}$ satisfying $WM=0$, we have $M=0$.

Let $\delta$ be a linear mapping from $\mathcal{A}$ into a left $\mathcal{A}$-module $\mathcal{M}$. $\delta$ is called a \emph{left derivation} if $\delta(AB)=A\delta(B)+B\delta(A)$ for each $A,B$ in $\mathcal{A}$, and is called a \emph{Jordan left derivation} if $\delta(A\circ B)=2A\delta(B)+2B\delta(A)$ for each $A,B$ in $\mathcal{A}$. Let $C\in\mathcal{A}$. $\delta$ is said to be \emph{left derivable at $C$} if $\delta(AB)=A\delta(B)+B\delta(A)$ for each $A,B$ in $\mathcal{A}$ satisfying $AB=C$, and to be \emph{Jordan left derivable at $C$} if $\delta(A\circ B)=2A\delta(B)+2B\delta(A)$ for each $A,B$ in $\mathcal{A}$ satisfying $A\circ B=C$.

The concepts of left derivations and Jordan left derivations are introduced by Bresar and Vukman in \cite{Bresar Vukman}. For results concerning the relationship between left derivations and Jordan left derivations on prime rings, we refer the reader to \cite{Ashraf Rehman Ali, Bresar Vukman, Vukman semiprime-ring, Vukman ring Banach-algebra}.
It's natural that every (Jordan) left derivation is (Jordan) left derivable at each point.
There have been a number of papers concerning the study of conditions under which (Jordan) derivations can be completely determined by the action on some sets of points \cite{Alaminos Bresar others C*, Alaminos Bresar others C* through 0, Alaminos Bresar others, Chebotar Ke Lee 0, Chebotar Ke Lee Zhang 0 Jordan, Ghahramani nontrivial-id, Hou Zhao 0, Zhu Xiong nest, Zhu Xiong reflexive}. Using the techniques of researching (Jordan) derivations, several authors are devoted to study the linear (or additive) mappings on some algebras behaving like (Jordan) left derivations when acting on special products. In \cite{Li Zhou OM}, Li and Zhou study left derivable mappings at zero on some algebras. In \cite{Ebadiana Gordji, Fadaee Ghahramani reflexive id}, the authors consider a continuous mapping $\delta$ satisfying $A\delta(A^{-1})+A^{-1}\delta(A)=0$ for each invertible element $A$ in von Neumann algebras or Banach algebras. In \cite{Fadaee Ghahramani reflexive id, Ghahramani reflexive id, Li Zhou OM}, the authors characterize continuous (Jordan) left derivable mappings at the identity or non-trivial idempotents on some algebras. Let $\mathcal{R}$ be a 2-torsion free prime ring, and $\mathcal{U}$ be a Lie ideal of $\mathcal{R}$ satisfying $U^2\in\mathcal{U}$ for each $U\in\mathcal{U}$. Ashraf, Rehman and Ali\cite{Ashraf Rehman Ali} show that if an additive mapping $\delta$ on $\mathcal{R}$ satisfying $\delta(U^2)=2U\delta(U)$ for each $U$ in $\mathcal{U}$, then either $\mathcal{U}\in\mathcal{Z(R)}$ or $\delta(\mathcal{U})=(0)$. In this paper, we study (Jordan) left derivable mappings at zero or some non-zero points.

The paper is organized as follows.
In Section \ref{nonzero}, we characterize linear mappings (Jordan) left derivable at some non-zero points without continuity assumption. Let $\delta$ be a linear mapping from a unital Banach algebra $\mathcal{A}$ into a unital left $\mathcal{A}$-module $\mathcal{M}$, and $C$ be a non-zero element in $\mathcal{A}$. If $\delta$ is left derivable at $C$, then $C\cdot\delta$ is a Jordan left derivation. In addition, when $C\in\mathcal{Z}(\mathcal{A})$, we prove that if $\delta$ is Jordan left derivable at $C$, then $C\cdot\delta$ is also a Jordan left derivation. Let $I$ be the identity of $\mathcal{A}$, and $W$ in $\mathcal{Z}(\mathcal{A})$ be a left separating point of $\mathcal{M}$. As applications of the preceding results, we conclude that $\delta$ is a Jordan left derivation if and only if $\delta$ is (Jordan) left derivable at $I$ or $W$, which generalizes the corresponding results in \cite{Ebadiana Gordji, Fadaee Ghahramani reflexive id, Ghahramani reflexive id}.

In Section \ref{B}, we prove that if $\delta$ is a linear mapping from a Banach algebra $\mathcal{A}$ into a Banach left $\mathcal{A}$-module $\mathcal{M}$, and $\delta$ is a generalized (Jordan) left derivation (see Definition \ref{generalized_(J)LD}) with $\xi$ in $\mathcal{M}^{\ast\ast}$, then $\xi$ can be chosen in $\mathcal{M}$ when $\mathcal{A}$ is unital or when $\mathcal{M}$ is a dual left $\mathcal{A}$-module. Let $\mathcal{A}$ have the property ($\mathbb{B}$) (see Definition \ref{Prop_B}). We consider the relations between generalized (Jordan) left derivations and (Jordan) left derivable mappings at zero from $\mathcal{A}$ into $\mathcal{M}$.

\vspace{1em}
\section{(Jordan) left derivable mappings at some non-zero points}\label{nonzero}

In this section, we suppose that $\mathcal{A}$ is a Banach algebra with identity $I$ and $\mathcal{M}$ is a unital left $\mathcal{A}$-module, unless stated otherwise. We consider linear mappings (Jordan) left derivable at some non-zero points without continuity assumption.

At first, we consider a continuous mapping $\delta$ satisfying $A\delta(A^{-1})+A^{-1}\delta(A)=0$ for each invertible element $A$ in $\mathcal{A}$. Fadaee and Ghahramani \cite{Fadaee Ghahramani reflexive id} consider it when $\delta$ is continuous. In this paper, we use a different technique to consider it while $\delta$ is not necessarily continuous.Replacing \cite[Lemma 2.1]{Fadaee Ghahramani reflexive id} with the following Lemma \ref{nonzero_lemma1}, we have the assumption of continuity of $\delta$ is not necessary in \cite[Proposition 2.2 and Theorem 2.5]{Fadaee Ghahramani reflexive id}.

\begin{lemma}\label{nonzero_lemma1}
Let $\delta:\mathcal{A}\rightarrow\mathcal{M}$ be a linear mapping. If for each invertible element $A$ in $\mathcal{A}$, we have
$$A\delta(A^{-1})+A^{-1}\delta(A)=\delta(I),$$
then $\delta$ is a Jordan left derivation.
\end{lemma}

\begin{proof}[\bf Proof]
  By assumption, we have $\delta(I)+\delta(I)=\delta(I)$. Thus, $\delta(I)=0$.
  Let $A$ be invertible in $\mathcal{A}$. By $A\delta(A^{-1})+A^{-1}\delta(A)=\delta(I)=0$, it follows that
  \begin{equation}\label{nonzero_lemma1_2}
    \delta(A)=-A^2\delta(A^{-1}).
  \end{equation}
   Let $T\in\mathcal{A}$, $n\in\mathbb{N}^{+}$ with $n\geq\|T\|+1$, and $B=nI+T$. Then $B$ and $(I-B)$ are both invertible in $\mathcal{A}$. By (\ref{nonzero_lemma1_2}),
  \begin{small}
  \begin{align*}
    \delta(B)=&-B^2\delta(B^{-\!1})=-B^2\delta(B^{-\!1}(I-B)^2-B) \\
    =&-B^2\delta(B^{-\!1}(I-B)^2)+B^2\delta(B) \\
    =&B^2(B^{-\!1}(I-B)^2)^2\delta((I-B)^{-\!2}B)+B^2\delta(B) \\
    %~=&~(I-B)^4\delta((I-B)^{-\!2}-(I-B)^{-\!1})+B^2\delta(B) \\
    =&(I-B)^4\delta((I-B)^{-\!2})-(I-B)^2(I-B)^2\delta((I-B)^{-\!1})+B^2\delta(B) \\
    =&-\delta((I-B)^2)+(I-B)^2\delta(I-B)+B^2\delta(B) \\
    %~=&~2\delta(B)-\delta(B^2)-\delta(B)+2B\delta(B)-B^2\delta(B)+B^2\delta(B) \\
    =&\delta(B)-\delta(B^2)+2B\delta(B).
  \end{align*}
  \end{small}
  Thus, $\delta(B^2)=2B\delta(B)$. Since $B=nI+T$, we have that $\delta((nI+T)^2)=2(nI+T)\delta(nI+T)$, then
  $$\delta(T^2)=2T\delta(T)$$
  for each $T$ in $\mathcal{A}$. That is, $\delta$ is a Jordan left derivation.
\end{proof}

\begin{theorem}\label{nonzero_THC}
Let $\delta:\mathcal{A}\rightarrow\mathcal{M}$ be a linear mapping, and $C$ be a non-zero element in $\mathcal{A}$. If $\delta$ is left derivable at $C$, then $C\cdot\delta$ is a Jordan left derivation.
\end{theorem}

\begin{proof}[\bf Proof]
If $\delta$ is left derivable at $C$, then $\delta(C)=\delta(I\cdot C)=C\delta(I)+\delta(C)$. That is, $C\delta(I)=0$. Let $A$ be invertible in $\mathcal{A}$. We have that $\delta(C)=\delta(A\cdot(A^{-1}C))=A\delta(A^{-1}C)+A^{-1}C\delta(A)$. So
\begin{eqnarray}
    A^{-1}C\delta(A)=\delta(C)-A\delta(A^{-1}C), \label{nonzero_THC_1}\\
    \delta(A^{-1}C)=A^{-1}\delta(C)-A^{-2}C\delta(A). \label{nonzero_THC_2}
\end{eqnarray}
Let $T\in\mathcal{A}$, $n\in\mathbb{N}^{+}$ with $n\geq\|T\|+1$, and $B=nI+T$. Then $B$ and $(I-B)$ are both invertible in $\mathcal{A}$. By (\ref{nonzero_THC_1}) and (\ref{nonzero_THC_2}),
\begin{small}
\begin{align*}
    B^{-\!1}C\delta(B)=&\delta(C)-B\delta(B^{-\!1}C)=\delta(C)-B\delta(B^{-\!1}(I-B)C+C) \\
    =&(I-B)\delta(C)-B\delta(B^{-\!1}(I-B)C) \\
    =&(I-B)\delta(C)-B[B^{-\!1}(I-B)\delta(C)-(B^{-\!1}(I-B))^{2}C\delta((I-B)^{-\!1}B)] \\
    %~&=~(I-B)\delta(W)-(I-B)\delta(W)+(I-B)B^{-1}(I-B)W\delta((I-B)^{-1}-I) \\
    =&(I-B)B^{-\!1}(I-B)C\delta((I-B)^{-\!1}) \\
    =&(I-B)B^{-\!1}[\delta(C)-(I-B)^{-\!1}\delta((I-B)C)] \\
    %~&=~(B^{-1}-I)\delta(W)-B^{-1}\delta(W-BW) \\
    =&-\delta(C)+B^{-\!1}\delta(BC).
\end{align*}
\end{small}
Thus, $B^{-1}\delta(BC)=\delta(C)+B^{-1}C\delta(B)$. Then
$$\delta(BC)=B\cdot B^{-1}\delta(BC)=B\delta(C)+C\delta(B).$$
Since $B=nI+T$, $\delta(nC+TC)=(nI+T)\delta(C)+C\delta(nI+T)$, that is
\begin{equation}\label{nonzero_THC_4}
  \delta(TC)=T\delta(C)+C\delta(T)
\end{equation}
for each $T$ in $\mathcal{A}$. Then for each invertible element $A$ in $\mathcal{A}$, by (\ref{nonzero_THC_4}),
\begin{align*}
    \delta(C)=&A\delta(A^{-1}C)+A^{-1}C\delta(A) \\
    =&AA^{-1}\delta(C)+AC\delta(A^{-1})+A^{-1}C\delta(A) \\
    =&\delta(C)+AC\delta(A^{-1})+A^{-1}C\delta(A),
\end{align*}
that is, $AC\delta(A^{-1})+A^{-1}C\delta(A)=0=C\delta(I)$ for each invertible element $A$ in $\mathcal{A}$. Let $\tilde{\delta}=C\cdot\delta$. By Lemma \ref{nonzero_lemma1}, we have that $\tilde{\delta}$ is a Jordan left derivation.
\end{proof}

\begin{theorem}\label{nonzero_THCJ}
Let $\delta:\mathcal{A}\rightarrow\mathcal{M}$ be a linear mapping, and $C$ be a non-zero element in $\mathcal{Z}(\mathcal{A})$. If $\delta$ is Jordan left derivable at $C$, then $C\cdot\delta$ is a Jordan left derivation.
\end{theorem}

\begin{proof}[\bf Proof]
The proof is similar as the proof of Theorem \ref{nonzero_THC}. To complete, we give the brief sketch of the proof in the following.

Suppose that for each $A,B$ in $\mathcal{A}$ with $A\circ B=C$, $\delta(C)=2A\delta(B)+2B\delta(A)$. Take $A=I$ and $2B=C$, it's easy to obtain that $C\delta(I)=0$. Let $A$ be invertible in $\mathcal{A}$. Since $C\in\mathcal{Z}(\mathcal{A})$ and $2\cdot\delta(C)=\delta(A\cdot(A^{-1}C)+(A^{-1}C)\cdot A)$, we have that $\delta(C)=A\delta(A^{-1}C)+A^{-1}C\delta(A)$. So
  \begin{eqnarray}
    A^{-1}C\delta(A)=\delta(C)-A\delta(A^{-1}C), \label{nonzero_THCJ_2}\\
    \delta(A^{-1}C)=A^{-1}\delta(C)-A^{-2}C\delta(A). \label{nonzero_THCJ_3}
  \end{eqnarray}
Let $T\in\mathcal{A}$, $n\in\mathbb{N}^{+}$ with $n\geq\|T\|+1$, and $B=nI+T$. Then $B$ and $(I-B)$ are both invertible in $\mathcal{A}$. Using (\ref{nonzero_THCJ_2}) and (\ref{nonzero_THCJ_3}), through a series of calculations, we obtain that
$$B^{-1}\delta(BC)=\delta(C)+B^{-1}C\delta(B).$$
Thus, $\delta(BC)=B\cdot B^{-1}\delta(BC)=B\delta(C)+C\delta(B)$. Since $B=nI+T$ and $C\delta(I)=0$, we obtain that $\delta(TC)=T\delta(C)+C\delta(T)$ for each $T$ in $\mathcal{A}$. Then for each invertible element $A$ in $\mathcal{A}$, $\delta(C)=A\delta(A^{-1}C)+A^{-1}C\delta(A)=\delta(C)+AC\delta(A^{-1})+A^{-1}C\delta(A)$, that is, $AC\delta(A^{-1})+A^{-1}C\delta(A)=0=C\delta(I)$. Let $\tilde{\delta}=C\cdot\delta$. By Lemma \ref{nonzero_lemma1}, we have that $\tilde{\delta}$ is a Jordan left derivation.
\end{proof}

By Theorems \ref{nonzero_THC} and \ref{nonzero_THCJ}, it's not difficult to show the following results.

\begin{theorem}\label{nonzero_THI}
Let $\delta:\mathcal{A}\rightarrow\mathcal{M}$ be a linear mapping. Then the following conditions are equivalent:
\vspace{-1ex}
\begin{enumerate}[\indent $\mathrm{(}i\mathrm{)}$]
  \item $\delta$ is a Jordan left derivation,
  \item $\delta$ is left derivable at $I$,
  \item $\delta$ is Jordan left derivable at $I$.
\end{enumerate}
\end{theorem}

\begin{proof}[\bf Proof]
  we only need to prove ($i$)$\Rightarrow$($ii$). If $\delta$ is a Jordan left derivation, then by $\delta(I)=2\delta(I)$, we have $\delta(I)=0.$
  For each $A,B$ in $\mathcal{A}$ with $AB=I$, we have $\delta(BA)=\delta(A\circ B)=2A\delta(B)+2B\delta(A)$. By
  \begin{align*}
    2B\delta(B)=&\delta(B^2)=\delta(B(BA)B) \\
    =&B^2\delta(BA)+3B(BA)\delta(B)-(BA)B\delta(B) \\
    =&B^2(2A\delta(B)+2B\delta(A))+3B^2A\delta(B)-B\delta(B) \\
    =&5B^2A\delta(B)+2B^3\delta(A)-B\delta(B),
  \end{align*}
  we have $3B\delta(B)=5B^2A\delta(B)+2B^3\delta(A)$. Then
  \begin{align*}
    3A\delta(B)=&3A^2\cdot B\delta(B)=A^2\cdot(5B^2A\delta(B)+2B^3\delta(A)) \\
    =&5A\delta(B)+2B\delta(A).
  \end{align*}
  Thus, $A\delta(B)+B\delta(A)=0=\delta(AB)$ for each $A,B$ in $\mathcal{A}$ with $AB=I$. That is, $\delta$ is left derivable at $I$.
\end{proof}

\begin{corollary}\label{nonzero_CoXY}
Let $\delta:\mathcal{A}\rightarrow\mathcal{M}$ be a linear mapping, and $X,Y$ be elements in $\mathcal{A}$ with $X+Y=I$. If $\delta$ is left derivable at $X$ and $Y$, then $\delta$ is a Jordan left derivation.\\
\indent In addition, when $X,Y\in\mathcal{Z}(\mathcal{A})$, then $\delta$ is a Jordan left derivation if and only if $\delta$ is Jordan left derivable at $X$ and $Y$.
\end{corollary}

\begin{theorem}\label{nonzero_THW}
Let $\delta:\mathcal{A}\rightarrow\mathcal{M}$ be a linear mapping, and $W$ in $\mathcal{Z}(\mathcal{A})$ be a left separating point of $\mathcal{M}$. Then the following conditions are equivalent:
\vspace{-1ex}
\begin{enumerate}[\indent $\mathrm{(}i\mathrm{)}$]
  \item $\delta$ is a Jordan left derivation,
  \item $\delta$ is left derivable at $W$,
  \item $\delta$ is Jordan left derivable at $W$.
\end{enumerate}
\end{theorem}

\begin{proof}[\bf Proof]
  It's obvious that $\delta$ is Jordan left derivable at $W$ if $\delta$ is a Jordan left derivation. We only need to prove ($i$)$\Rightarrow$($ii$), ($ii$)$\Rightarrow$($i$) and ($iii$)$\Rightarrow$($i$).

  If $\delta$ is left derivable at $W$, or if $\delta$ is Jordan left derivable at $W$, then by Theorems \ref{nonzero_THC} and \ref{nonzero_THCJ}, $W\cdot\delta$ is Jordan left derivation. For each $A$ in $\mathcal{A}$, $$W\delta(A^2)=2AW\delta(A).$$ Since $W\in\mathcal{Z}(\mathcal{A})$, we have $W\delta(A^2)=W\cdot 2A\delta(A)$. Since $W$ is a left separating point of $\mathcal{M}$, we obtain that $\delta(A^2)=2A\delta(A)$ for each $A$ in $\mathcal{A}$. That is, $\delta$ is a Jordan left derivation.

  If $\delta$ is a Jordan left derivation, then for each $B$ in $\mathcal{A}$, by $BW=WB$,
  $$\delta(BW)=\delta(WB)=B\delta(W)+W\delta(B).$$
  On the other hand, for each $A,B$ in $\mathcal{A}$ with $AB=W$, we have that $$\delta(BW)=\delta(BAB)=B^2\delta(A)+3BA\delta(B)-AB\delta(B).$$
  Thus
  $$B\delta(AB)+2AB\delta(B)=B^2\delta(A)+3BA\delta(B)$$
  for each $A,B$ in $\mathcal{A}$ with $AB=W$. Multiply the equation by $A$ at the left side,
  $$AB\delta(AB)+2AAB\delta(B)=AB^2\delta(A)+3ABA\delta(B).$$
  Since $W=AB\in\mathcal{Z}(\mathcal{A})$, we have $W\delta(AB)=WB\delta(A)+WA\delta(B)$. And since $W$ is a left separating point of $\mathcal{M}$, we obtain that $\delta(AB)=B\delta(A)+A\delta(B)$ for each $A,B$ in $\mathcal{A}$ with $AB=W$, i.e., $\delta$ is left derivable at $W$.
\end{proof}

By \cite{Li Zhou OM, Vukman ring Banach-algebra}, if $\mathcal{A}$ is a CSL algebra or a unital semisimple Banach algebra, then  every continuous Jordan left derivation on $\mathcal{A}$ is zero. So from Theorems \ref{nonzero_THC} and \ref{nonzero_THCJ} we have the following corollaries.

\begin{corollary}\label{nonzero_CSL_semisimpleI}
Let $\mathcal{A}$ be a CSL algebra or a unital semisimple Banach algebra, and $\delta$ be a continuous linear mapping on $\mathcal{A}$. If $\delta$ is left derivable at $I$, or if $\delta$ is Jordan left derivable at $I$, then $\delta=0$.
\end{corollary}

\begin{corollary}\label{nonzero_CSL_semisimpleW}
Let $\mathcal{A}$ be as in Corollary \ref{nonzero_CSL_semisimpleI}, $\delta$ be a continuous linear mapping on $\mathcal{A}$, and $W$ in $\mathcal{A}$ be a left separating point of $\mathcal{A}$. If $\delta$ is left derivable at $W$, then $\delta=0$.\\
\indent In addition, when $W\in\mathcal{Z}(\mathcal{A})$, if $\delta$ is Jordan left derivable at $W$, then $\delta=0$.
\end{corollary}

\begin{remark}
Let $\mathcal{A}$ be a unital Banach algebra, $\mathcal{M}$ be a unital Banach left $\mathcal{A}$-module, and $W$ in $\mathcal{Z(A)}$ be a left separating point of $\mathcal{M}$. Fadaee and Ghahramani \cite{Fadaee Ghahramani reflexive id} prove that a continuous linear mapping $\delta:\mathcal{A}\rightarrow\mathcal{M}$ is a Jordan left derivation if $A\delta(A^{-1})+A^{-1}\delta(A)=0$ for each invertible element $A$ in $\mathcal{A}$, or if $\delta$ is Jordan left derivable at $I$. Ghahramani\cite{Ghahramani reflexive id} prove that every continuous linear mapping $\delta:\mathcal{A}\rightarrow\mathcal{M}$ is a Jordan left derivation if $\delta$ is left derivable at $I$. Ebadiana and Gordji\cite{Ebadiana Gordji} prove that a bounded linear mapping $\delta:\mathcal{A}\rightarrow\mathcal{M}$ is a Jordan left derivation if $\delta(AB)=A\delta(B)+B\delta(A)$ for each $A,B$ in $\mathcal{A}$ with $AB=BA=W$.
In this Section, we improve the results in\cite{Ebadiana Gordji, Fadaee Ghahramani reflexive id, Ghahramani reflexive id} without assume that $\delta$ is bounded or continuous. Correspondingly, we conclude that every linear mapping $\delta:\mathcal{A}\rightarrow\mathcal{M}$ is a Jordan left derivation if $A\delta(A^{-1})+A^{-1}\delta(A)=0$ for each invertible element $A$ in $\mathcal{A}$, or if $\delta$ is (Jordan) left derivable at $I$ or $W$.

Let $\mathcal{R}$ be a 2-torsion free ring with identity $I$ which satisfies that for each $T$ in $\mathcal{R}$, there is some integer $n$ such that $nI-T$ and $(n-1)I-T$ are invertible or $nI+T$ and $(n+1)I+T$ are invertible. If we replace $\mathcal{A}$ with $\mathcal{R}$ and replace linear mappings with additive mappings, then all of the above results in this section are still true.
\end{remark}

\vspace{1em}
\section{(Jordan) left derivable mappings at zero} \label{B}

In this section, we consider continuous linear operators (Jordan) left derivable at zero. At first, we introduce a class of Banach algebras with the property ($\mathbb{B}$). Let $\mathcal{A}$ be a Banach algebra and $\phi$ be a continuous bilinear mapping from $\mathcal{A}\times\mathcal{A}$ into a Banach space $\mathcal{X}$. We say that $\phi$ \emph{preserves zero products} if for each $A,B$ in $\mathcal{A}$,
$$AB=0~\Rightarrow~\phi(A,B)=0.$$
Then the property ($\mathbb{B}$) is defined as follows.

\begin{definition}\cite{Alaminos Bresar others}\label{Prop_B}
A Banach algebra $\mathcal{A}$ is saied to have \emph{the property ($\mathbb{B}$)}, if for every Banach space $\mathcal{X}$ and every continuous bilinear mapping $\phi : \mathcal{A}\times\mathcal{A}\rightarrow\mathcal{X}$, $\phi$ preserving zero products implies that for each $A,B,C$ in $\mathcal{A}$, $$\phi(AB,C)=\phi(A,BC).$$
\end{definition}
The class of Banach algebras with the property ($\mathbb{B}$) includes $C^{\ast}$-algebras, group-algebras, unitary algebras, and Banach algebras generated by idempotents.

Let $\mathcal{A}$ be a Banach algebra, and $\mathcal{M}$ be a Banach left $\mathcal{A}$-module. $\mathcal{M}$ is called \emph{essential} if $\mathcal{M}=\overline{span\{AM:A\in\mathcal{A},M\in\mathcal{M}\}}$. Let $\mathcal{M}^{\ast}$ be the dual space of $\mathcal{M}$. For each $A$ in $\mathcal{A}$ and each $M^{\ast}$ in $\mathcal{M}^{\ast}$, we define the module multiplication $M^{\ast}A$ in $\mathcal{M}^{\ast}$ by
$$\langle M, M^{\ast}A\rangle=\langle AM, M^{\ast}\rangle\quad(\forall M\in\mathcal{M}).$$
Then $\mathcal{M}^{\ast}$ is a Banach right $\mathcal{A}$-module called the \emph{dual} Banach right $\mathcal{A}$-module. Similarly each Banach right $\mathcal{A}$-module has a dual Banach left $\mathcal{A}$-module.
Note that $\mathcal{M}^{\ast\ast}$ is a Banach left $\mathcal{A}$-module, and we may regard $\mathcal{M}$ as a closed submodule of $\mathcal{M}^{\ast\ast}$. In fact, the natural embedding $\iota_{\mathcal{M}} : \mathcal{M}\rightarrow\mathcal{M}^{\ast\ast}$ is a left $\mathcal{A}$-module homomorphism and the natural projection $\pi_{\mathcal{M}} : \mathcal{M}^{\ast\ast\ast}\rightarrow\mathcal{M}^{\ast}$ is a right $\mathcal{A}$-module homomorphism. $\pi_{\mathcal{M}}$ is the dual operator $\iota_{\mathcal{M}}^{\ast}$ of $\iota_{\mathcal{M}}$.

Next, we consider the following conditions on a continuous linear mapping $\delta$ from a Banach algebra $\mathcal{A}$ into a Banach left $\mathcal{A}$-module $\mathcal{M}$:
\vspace{-1ex}
\begin{enumerate}[\indent (D1)]
  \item for each $A,B,C$ in $\mathcal{A},~AB=BC=0~\Rightarrow~AC\delta(B)=0,$
  \item for each $A,B$ in $\mathcal{A},~AB=0~\Rightarrow~A\delta(B)+B\delta(A)=0,$
  \item for each $A,B$ in $\mathcal{A},~AB=BA=0~\Rightarrow~A\delta(B)+B\delta(A)=0,$
  \item for each $A,B$ in $\mathcal{A},~A\circ B=0~\Rightarrow~A\delta(B)+B\delta(A)=0.$
\end{enumerate}

Similar to the definition of generalized derivations in \cite{Alaminos Bresar others}, generalized (Jordan) left derivations are defined as follows.
\begin{definition}\label{generalized_(J)LD}
Let $\delta$ be a linear mapping from a Banach algebra $\mathcal{A}$ into a Banach left $\mathcal{A}$-module $\mathcal{M}$. Then $\delta$ is called
\vspace{-0.5ex}
\begin{enumerate}[($i$)]
  \item a \emph{generalized left derivation}, if there exists an element $\xi$ in $\mathcal{M}^{\ast\ast}$, such that for each $A,B$ in $\mathcal{A}$,
      $$\delta(AB)=A\delta(B)+B\delta(A)-AB\xi.$$
  \item a \emph{generalized Jordan left derivation}, if there exists an element $\xi$ in $\mathcal{M}^{\ast\ast}$, such that for each $A,B$ in $\mathcal{A}$,
      $$\delta(A\circ B)=2A\delta(B)+2B\delta(A)-(A\circ B)\xi.$$
\end{enumerate}
\end{definition}
\vspace{1ex}
\begin{remark}
If $\delta$ is a generalized left derivation, then there exists a mapping $d:\mathcal{A}\rightarrow\mathcal{M}^{\ast\ast}$ satisfying $d(A)=\delta(A)-A\xi$ for each $A$ in $\mathcal{A}$. We denote that $R_{\xi}(A)=A\xi$ for each $A$ in $\mathcal{A}$, then
$$\delta=d+R_{\xi}.$$
We cannot confirm that $d$ is a left derivation since $\mathcal{A}$ is not necessarily commutative.
But if $\delta$ is a generalized Jordan left derivation, and $d:\mathcal{A}\rightarrow\mathcal{M}^{\ast\ast}$ is a mapping satisfying
$$\delta=d+R_{\xi},$$
then $d$ is a Jordan left derivation.
\end{remark}

In the next four propositions, we establish several sufficient conditions which imply that $\xi$ can be chosen in $\mathcal{M}$.

\begin{proposition}\label{B_prop1}
Let $\mathcal{A}$ be a Banach algebra with identity $I$, $\mathcal{M}$ be an arbitrary Banach left $\mathcal{A}$-module, and $\delta:\mathcal{A}\rightarrow\mathcal{M}$ be a generalized left derivation. Then there exists an element $\xi$ in $\mathcal{M}$ such that for each $A,B$ in $\mathcal{A}$,
$$\delta(AB)=A\delta(B)+B\delta(A)-AB\xi.$$
Furthermore,
\vspace{-1ex}
\begin{center}
 $\delta$ is a left derivation~~$\Leftrightarrow$~~$I\cdot\delta(I)=0$.
\end{center}
\end{proposition}

\begin{proof}[\bf Proof]
According to the definition of generalized left derivation, there exists an element $\zeta$ in $\mathcal{M}^{\ast\ast}$ such that for each $A,B$ in $\mathcal{A}$,
\begin{equation}\label{B_prop1_1}
\delta(AB)=A\delta(B)+B\delta(A)-AB\zeta.
\end{equation}
Note this entails that $AB\zeta\in\mathcal{M}$ for each $A,B$ in $\mathcal{A}$. Define $\xi=I\zeta\in\mathcal{M}$, then on account of (\ref{B_prop1_1}), we have that for each $A,B$ in $\mathcal{A}$,
\begin{equation}\label{B_prop1_2}
\delta(AB)=A\delta(B)+B\delta(A)-AB\xi.
\end{equation}

Furthermore, if $\delta$ is a left derivation, then $\delta(I)=I\cdot\delta(I)+I\cdot\delta(I)$, i.e., $I\cdot\delta(I)=0$. And if $I\cdot\delta(I)=0$, then taking $A=B=I$ in (\ref{B_prop1_2}), we have $\xi=I\cdot\delta(I)=0$. So $\delta(AB)=A\delta(B)+B\delta(A)$ for each $A,B$ in $\mathcal{A}$.
\end{proof}

According to Cohen's factorization theorem in \cite{Bonsall Duncan}, if a Banach algebra $\mathcal{A}$ has a bounded left approximate identity, then for each $C$ in $\mathcal{A}$, there exist elements $A,B$ in $\mathcal{A}$ such that $C=AB$.

\begin{proposition}\label{B_prop2}
Let $\mathcal{A}$ be a Banach algebra, $\mathcal{M}$ be a dual Banach left $\mathcal{A}$-module, and $\delta:\mathcal{A}\rightarrow\mathcal{M}$ be a generalized left derivation. Then there exists an element $\xi$ in $\mathcal{M}$ such that for each $A,B$ in $\mathcal{A}$,
$$\delta(AB)=A\delta(B)+B\delta(A)-AB\xi.$$
If $\mathcal{A}$ has a bounded approximate identity $(\rho_i)_{i\in \mathcal{I}}$ and $\delta$ is continuous, then
\begin{center}
$\mathrm{(}$i$\mathrm{)}~\delta$ is a left derivation~~$\Leftrightarrow~~\mathrm{(}$ii$\mathrm{)}~\lim\limits_{i\in \mathcal{I}}~C\cdot\delta(\rho_i)=0$~for~each~$C$~in~$\mathcal{A}$.
\end{center}
In addition, if $Rann_{\mathcal{M}}(\mathcal{A})=\{0\}$, and $\mathcal{M}_{\ast}$ is a predual right $\mathcal{A}$-module of $\mathcal{M}$, then
\begin{center}
$\mathrm{(}$i$\mathrm{)}~\Leftrightarrow~\mathrm{(}$ii$\mathrm{)}~\Leftrightarrow~\mathrm{(}$iii$\mathrm{)}~\sigma(\mathcal{M},\mathcal{M}_{\ast})\!-\!\lim\limits_{i\in \mathcal{I}}\delta(\rho_i)=0$.
\end{center}
\end{proposition}

\begin{proof}[\bf Proof]
According to the definition of generalized left derivation, there exists an element $\zeta$ in $\mathcal{M}^{\ast\ast}$ such that for each $A,B$ in $\mathcal{A}$,
\begin{equation}\label{B_prop2_1}
\delta(AB)=A\delta(B)+B\delta(A)-AB\zeta.
\end{equation}
Let $\mathcal{M}_{\ast}$ be a predual right $\mathcal{A}$-module of $\mathcal{M}$, $\pi_{\mathcal{M}_{\ast}}:\mathcal{M}^{\ast\ast}\rightarrow\mathcal{M}$ be the natural projection. Set $\xi=\pi_{\mathcal{M}_{\ast}}(\zeta)\in\mathcal{M}$. Since $\pi_{\mathcal{M}_{\ast}}$ is a left $\mathcal{A}$-module homomorphism, applying $\pi_{\mathcal{M}_{\ast}}$ to (\ref{B_prop2_1}), we arrive at that for each $A,B$ in $\mathcal{A}$,
\begin{equation}\label{B_prop2_2}
 \delta(AB)=A\delta(B)+B\delta(A)-AB\xi.
\end{equation}

If $\mathcal{A}$ has a bounded approximate identity $(\rho_i)_{i\in \mathcal{I}}$ and $\delta$ is continuous, then by (\ref{B_prop2_2}), it follows that
\begin{align*}
  A\cdot\delta(B) & = A\cdot\lim\limits_{i\in \mathcal{I}}\delta(B\rho_i) \\
  & = \lim\limits_{i\in \mathcal{I}}(AB\delta(\rho_i)+A\rho_i\delta(B)-AB\rho_i\xi) \\
  & = \lim\limits_{i\in \mathcal{I}}AB\delta(\rho_i)+A\delta(B)-AB\xi.
\end{align*}
Thus, $\lim\limits_{i\in \mathcal{I}}AB\delta(\rho_i)=AB\xi$ for each $A,B$ in $\mathcal{A}$%, i.e.
. According to Cohen's factorization theorem, we have that
\begin{equation}\label{B_prop2_3}
\lim\limits_{i\in \mathcal{I}}~C\cdot\delta(\rho_i)=C\cdot\xi
\end{equation}
for each $C$ in $\mathcal{A}$.

\noindent\hangafter=1\setlength{\hangindent}{2em}
\item[~($i$)$\Rightarrow$($ii$)] If $\delta$ is a left derivation, then by (\ref{B_prop2_2}), we have $AB\xi=0$ for each $A,B$ in $\mathcal{A}$, i.e., $C\xi=0$ for each $C$ in $\mathcal{A}$, and then by (\ref{B_prop2_3}), we arrive at that $$\lim\limits_{i\in \mathcal{I}}~C\cdot\delta(\rho_i)=0$$ for each $C$ in $\mathcal{A}$.
\item[~($ii$)$\Rightarrow$($i$)] If $\lim\limits_{i\in \mathcal{I}}~C\cdot\delta(\rho_i)=0$ for each $C$ in $\mathcal{A}$, then by (\ref{B_prop2_3}), we have $C\xi=0$ for each $C$ in $\mathcal{A}$, and then by (\ref{B_prop2_2}), we arrive
    at that $$\delta(AB)=A\delta(B)+B\delta(A)$$ for each $A,B$ in $\mathcal{A}$, i.e., $\delta$ is a left derivation.

\setlength{\hangindent}{0em}\indent
In addition, if $Rann_{\mathcal{M}}(\mathcal{A})=\{0\}$, we claim that $\mathcal{M}_{\ast}=\overline{\mathcal{M}_{\ast}\mathcal{A}}$. For otherwise, there exists an element $M$ in $\mathcal{M}\setminus\{0\}$ such that $\langle M_{\ast}C,M\rangle=0$ for each $C$ in $\mathcal{A}$ and $M_{\ast}$ in $\mathcal{M}_{\ast}$. It follows that $\langle M_{\ast},CM\rangle=0$ for each $C$ in $\mathcal{A}$ and $M_{\ast}$ in $\mathcal{M}_{\ast}$, and so $CM=0$ for each $C$ in $\mathcal{A}$. Since $Rann_{\mathcal{M}}(\mathcal{A})=\{0\}$, then $M=0$ which contradicts that $M\in\mathcal{M}\setminus\{0\}$. So $\mathcal{M}_{\ast}=\overline{\mathcal{M}_{\ast}\mathcal{A}}$.

\noindent\hangafter=1\setlength{\hangindent}{2em}
\item[~($ii$)$\Rightarrow$($iii$)] Suppose that $\lim\limits_{i\in \mathcal{I}}~C\cdot\delta(\rho_i)=0$ for each $C$ in $\mathcal{A}$. For each $M_{\ast}$ in $\mathcal{M}_{\ast}$, there exist $\alpha_j$ in $\mathbb{C}$, $(N_{\ast})_j$ in $\mathcal{M}_{\ast}$, and $C_j$ in $\mathcal{A}$ $(n=1,2,...,n)$ such that $M_{\ast}=\sum\limits_{j=1}^{n} \alpha_j(N_{\ast})_j C_j$. Thus
    \begin{small}
    \begin{align*}
      \lim\limits_{i\in \mathcal{I}}~\langle M_{\ast}, \delta(\rho_i)\rangle &= \lim\limits_{i\in \mathcal{I}}~\langle \sum\limits_{j=1}^{n} \alpha_j(N_{\ast})_j C_j, \delta(\rho_i)\rangle\\
      &= \lim\limits_{i\in \mathcal{I}}~\sum\limits_{j=1}^{n} \alpha_j\langle (N_{\ast})_j C_j, \delta(\rho_i)\rangle = \lim\limits_{i\in \mathcal{I}}~\sum\limits_{j=1}^{n} \alpha_j\langle (N_{\ast})_j, C_j\delta(\rho_i)\rangle.
    \end{align*}
    \end{small}
    Since $\{\|\delta(\rho_i)\|:i\in \mathcal{I}\}$ is bounded, we have that for each $M_{\ast}$ in $\mathcal{M}_{\ast}$, $\lim\limits_{i\in \mathcal{I}}~\langle M_{\ast},\delta(\rho_i)\rangle=0$. So (iii) holds.
\item[~($iii$)$\Rightarrow$($i$)] If $\sigma(\mathcal{M},\mathcal{M}_{\ast})-\lim\limits_{i\in \mathcal{I}}\delta(\rho_i)=0$, then by (\ref{B_prop2_3}), for each $C$ in $\mathcal{A}$ and $M_{\ast}$ in $\mathcal{M}_{\ast}$,
      $$\langle M_{\ast},C\cdot\xi\rangle=\lim\limits_{i\in \mathcal{I}}~\langle M_{\ast},C\cdot\delta(\rho_i)\rangle=\lim\limits_{i\in \mathcal{I}}~\langle M_{\ast}\cdot C,\delta(\rho_i)\rangle=0.$$
      Thus $C\cdot\xi=0$ for each $C$ in $\mathcal{A}$. With (\ref{B_prop2_2}),
      $$\delta(AB)=A\delta(B)+B\delta(A)$$
      for each $A,B$ in $\mathcal{A}$. That is, $\delta$ is a left derivation.
\end{proof}

The proof of the following theorem is similar to the proof of Proposition \ref{B_prop1}. We leave it to the reader.
\begin{proposition}\label{B_prop3}
Let $\mathcal{A}$ and $\mathcal{M}$ be as in Proposition \ref{B_prop1},
and $\delta:\mathcal{A}\rightarrow\mathcal{M}$ be a generalized Jordan left derivation. Then there exists an element $\xi$ in $\mathcal{M}$ such that for each $A,B$ in $\mathcal{A}$,
$$\delta(A\circ B)=2A\delta(B)+2B\delta(A)-(A\circ B)\xi.$$
Furthermore,
\vspace{-1ex}
\begin{center}
$\delta$ is a Jordan left derivation~~$\Leftrightarrow$~~$I\cdot\delta(I)=0$.
\end{center}
\end{proposition}
\vspace{1ex}
\begin{proposition}\label{B_prop4}
Let $\mathcal{A}$ and $\mathcal{M}$ be as in Proposition \ref{B_prop2},
and $\delta:\mathcal{A}\rightarrow\mathcal{M}$ be a generalized Jordan left derivation. Then there exists an element $\xi$ in $\mathcal{M}$ such that for each $A,B$ in $\mathcal{A}$,
\begin{equation}\label{B_prop4_0}
\delta(A\circ B)=2A\delta(B)+2B\delta(A)-(A\circ B)\xi.
\end{equation}
If $\mathcal{A}$ has a bounded approximate identity $(\rho_i)_{i\in \mathcal{I}}$ and $\delta$ is continuous, then
\begin{center}
$\delta$ is a Jordan left derivation~~$\Leftrightarrow$~~$\lim\limits_{i\in \mathcal{I}}~A^2\cdot\delta(\rho_i)=0$ for each $A$ in $\mathcal{A}$.
\end{center}
\end{proposition}

\begin{proof}[\bf Proof]
The existence of $\xi\in\mathcal{M}$ can be proved similarly as the proof given in Proposition \ref{B_prop2}.

If $\mathcal{A}$ has a bounded approximate identity $(\rho_i)_{i\in \mathcal{I}}$ and $\delta$ is continuous, then
with the following equation
\begin{equation}\label{B_prop4_1}
2A\rho_i A=(A\circ(A\circ\rho_i))-A^2\circ\rho_i
\end{equation}
for each $A$ in $\mathcal{A}$, we obtain that
\begin{align*}
  2\delta(A\rho_i A)~=&~\delta(A\circ(A\circ\rho_i))-\delta(A^2\circ\rho_i)\\
  =&~2A^2\delta(\rho_i)+(6A\rho_i+2\rho_i A)\delta(A)-2\rho_i\delta(A^2)-(2A^2\rho_i+4A\rho_i A)\xi.
\end{align*}
By taking limits in the above identities, we arrive at that
\begin{align*}
   2\delta(A^2)~=&~\lim\limits_{i\in \mathcal{I}}~2\delta(A\rho_i A)\\
   ~=&~\lim\limits_{i\in \mathcal{I}}~(2A^2\delta(\rho_i)+(6A\rho_i+2\rho_i A)\delta(A)-2\rho_i\delta(A^2)-(2A^2\rho_i+4A\rho_i A)\xi)\\
   ~=&~2\lim\limits_{i\in \mathcal{I}}A^2\delta(\rho_i)+8A\delta(A)-2\delta(A^2)-6A^2\xi.
\end{align*}
Since $2\delta(A^2)=\delta(A\circ A)=4A\delta(A)-2A^2\xi$, we obtain that
\begin{equation}\label{B_prop4_2}
\lim\limits_{i\in \mathcal{I}}A^2\delta(\rho_i)=A^2\xi
\end{equation}
for each $A$ in $\mathcal{A}$.

Conversely, if $\delta$ is a Jordan left derivation, by (\ref{B_prop4_0}), then $(A\circ B)\xi=0$ for each $A,B$ in $\mathcal{A}$. With (\ref{B_prop4_2}), we have that for each $A$ in $\mathcal{A}$
$$\lim\limits_{i\in \mathcal{I}}A^2\delta(\rho_i)=A^2\xi=0.$$
And if $\lim\limits_{i\in \mathcal{I}}~A^2\cdot\delta(\rho_i)=0$ for each $A$ in $\mathcal{A}$, then by (\ref{B_prop4_2}), we have $A^2\xi=0$ for each $A$ in $\mathcal{A}$. With (\ref{B_prop4_0}), we have that $\delta(A^2)=2A\delta(A)$ for each $A$ in $\mathcal{A}$. That is, $\delta$ is a Jordan left derivation.
\end{proof}

Now, we arrive at our main goals on the basis of the above results.

\begin{theorem}\label{B_TH1}
Let $\mathcal{A}$ be a Banach algebra with the property ($\mathbb{B}$) and have a bounded approximate identity, $\mathcal{M}$ be a Banach left $\mathcal{A}$-module with $Rann_{\mathcal{M}}(\mathcal{A})=\{0\}$, and $\delta:\mathcal{A}\rightarrow\mathcal{M}$ be a continuous linear operator. Then $\delta$ satisfies (D1) if and only if for each $C$ in $\mathcal{A}$, $C\cdot\delta$ is a generalized left derivation.\\
\indent Furthermore, let $C$ be a element in $\mathcal{A}$. If $C\cdot\delta$ is a generalized left derivation with $\xi$ in $\mathcal{M}^{\ast\ast}$, then $\xi$ can be chosen in $\mathcal{M}$ in each of the following cases:
\vspace{-1ex}
\begin{enumerate}[\indent $\mathrm{(}i\mathrm{)}$]
  \item $\mathcal{A}$ is unital,
  \item $\mathcal{M}$ is a dual left $\mathcal{A}$-module.
\end{enumerate}
\end{theorem}

\begin{proof}[\bf Proof]
If $\delta$ satisfies (D1), then for each $A',B'$ in $\mathcal{A}$ with $A'B'=0$, we define the bilinear map $\phi~:~\mathcal{A}\times\mathcal{A}\rightarrow\mathcal{M}$ by $$\phi(A,B)=AB'\delta(BA')$$
for each $A,B$ in $\mathcal{A}$. For each $A,B$ in $\mathcal{A}$ with $AB=0$, since $ABA'=BA'B'=0$, we have that
$$\phi(A,B)=0.$$
Since $\mathcal{A}$ has the property ($\mathbb{B}$), we have that for each $A,B,C$ in $\mathcal{A}$, $\phi(AB,C)=\phi(A,BC)$, i.e.,
\begin{equation}\label{B_TH1_1}
ABB'\delta(CA')=AB'\delta(BCA').
\end{equation}
Now we fix $A,B,C$ in $\mathcal{A}$, and consider the bilinear mapping $\psi~:~\mathcal{A}\times\mathcal{A}\rightarrow\mathcal{M}$ such that for each $A',B'$ in $\mathcal{A}$,
$$\psi(A',B')=ABB'\delta(CA')-AB'\delta(BCA').$$
On account of (\ref{B_TH1_1}), $\psi(A',B')=0$ for each $A',B'$ in $\mathcal{A}$ with $A'B'=0$. Since $\mathcal{A}$ has the property ($\mathbb{B}$), it follows that $$\psi(A'B',C')=\psi(A',B'C')$$ for each $A',B',C'$ in $\mathcal{A}$. Thus, for each $A',B',C',A,B,C$ in $\mathcal{A}$,
\begin{multline}\label{B_TH1_2}
  ABC'\delta(CA'B')-AC'\delta(BCA'B')\\
  =ABB'C'\delta(CA')-AB'C'\delta(BCA').
\end{multline}
Denote $\tilde{\delta}=C'\delta$ for each $C'$ in $\mathcal{A}$, and $X=CA'$ %in $\mathcal{A}$
, then for each $X,B',A,B$ in $\mathcal{A}$,
$$A(\tilde{\delta}(BXB')-B\tilde{\delta}(XB')-B'\tilde{\delta}(BX)+BB'\tilde{\delta}(X))=0.$$
Since $Rann_{\mathcal{M}}(\mathcal{A})=\{0\}$, we have that for each $X,B',B$ in $\mathcal{A}$,
\begin{equation}\label{B_TH1_3}
BB'\tilde{\delta}(X)=B\tilde{\delta}(XB')+B'\tilde{\delta}(BX)-\tilde{\delta}(BXB').
\end{equation}

Let $(\rho_i)_{i\in \mathcal{I}}$ be a bounded approximate identity of $\mathcal{A}$. Since the net $(\delta(\rho_i))_{i\in \mathcal{I}}$ is bounded, it has a convergent subnet. Without loss of generality, we can assume that there exists an element $\xi$ in $\mathcal{M}^{\ast\ast}$ such that $\lim\limits_{i\in \mathcal{I}}\tilde{\delta}(\rho_i)=\lim\limits_{i\in \mathcal{I}}C'\delta(\rho_i)=\xi$ with respect to the topology $\sigma(\mathcal{M}^{\ast\ast},\mathcal{M}^{\ast})$. With (\ref{B_TH1_3}), we have that for each $B',B$ in $\mathcal{A}$,
$$BB'\tilde{\delta}(\rho_i)=B\tilde{\delta}(\rho_iB')+B'\tilde{\delta}(B\rho_i)-\tilde{\delta}(B\rho_iB').$$
Therefore, the net $(BB'\tilde{\delta}(\rho_i))_{i\in \mathcal{I}}$ satisfies $$\lim\limits_{i\in \mathcal{I}}~ BB'\tilde{\delta}(\rho_i)=B\tilde{\delta}(B')+B'\tilde{\delta}(B)-\tilde{\delta}(BB')$$ with respect to the norm topology. On the other hand, for each $M_{\ast}$ in $\mathcal{M}_{\ast}$, with respect to $\sigma(\mathcal{M}^{\ast\ast},\mathcal{M}^{\ast})$, we have that
\begin{align*}
  \langle M_{\ast},~\lim\limits_{i\in \mathcal{I}}~ BB'\tilde{\delta}(\rho_i)\rangle & =\lim\limits_{i\in \mathcal{I}}~ \langle M_{\ast},~ BB'\tilde{\delta}(\rho_i)\rangle=\lim\limits_{i\in \mathcal{I}}~ \langle M_{\ast}BB',~\tilde{\delta}(\rho_i)\rangle \\
   & =\lim\limits_{i\in \mathcal{I}}~ \langle M_{\ast}BB',~\xi \rangle=\lim\limits_{i\in \mathcal{I}}~ \langle M_{\ast},~ BB'\xi \rangle~~
\end{align*}
Thus there exists $\xi$ in $\mathcal{M}^{\ast\ast}$ such that for each $B',B$ in $\mathcal{A}$,
$$\tilde{\delta}(BB')=B\tilde{\delta}(B')+B'\tilde{\delta}(B)-BB'\xi,$$
i.e., $\tilde{\delta}$ is a generalized left derivation.

If for each $C$ in $\mathcal{A}$, $C\cdot\delta$ is a generalized left derivation, then for each $A,B,D$ in $\mathcal{A}$ with $AB=BD=0$, we have that
$$0=D\delta(AB)=AD\delta(B)+BD\delta(A)-AB\xi=AD\delta(B).$$
Hence, $AD\delta(B)=0$ for each $A,B,D$ in $\mathcal{A}$ with $AB=BD=0$.

Furthermore, let $C$ in $\mathcal{A}$. If $C\cdot\delta$ is a generalized left derivation with $\xi$ in $\mathcal{M}^{\ast\ast}$, then by Propositions \ref{B_prop1} and \ref{B_prop2}, $\xi$ can be chosen in $\mathcal{M}$ if $\mathcal{A}$ is unital, or if $\mathcal{M}$ is a dual left $\mathcal{A}$-module.
\end{proof}

\begin{theorem}\label{B_TH2}
Let $\mathcal{A}$ and $\mathcal{M}$ be as in Theorem \ref{B_TH1},
and $\delta:\mathcal{A}\rightarrow\mathcal{M}$ be a continuous linear operator. If $\delta$ satisfies (D2), then for each $C$ in $\mathcal{A}$, $C\cdot\delta$ is a generalized left derivation. \\
\indent Furthermore, let %$C\in\mathcal{A}$.
$C$ be a element in $\mathcal{A}$. If $C\cdot\delta$ is a generalized left derivation with $\xi$ in $\mathcal{M}^{\ast\ast}$, then $\xi$ can be chosen in $\mathcal{M}$ in each of the following cases:
\vspace{-1ex}
\begin{enumerate}[\indent $\mathrm{(}i\mathrm{)}$]
  \item $\mathcal{A}$ is unital,
  \item $\mathcal{M}$ is a dual left $\mathcal{A}$-module.
\end{enumerate}
\end{theorem}

\begin{proof}[\bf Proof]
Take $A,B,C$ in $\mathcal{A}$ satisfying $AB=BC=0$, then
$$0=A\cdot(B\delta(C)+C\delta(B))=AB\delta(C)+AC\delta(B)=AC\delta(B).$$
For each $C$ in $\mathcal{A}$, let $\tilde{\delta}=C\cdot\delta$. By Theorem \ref{B_TH1}, $\tilde{\delta}$ is a generalized left derivation with $\xi$ in $\mathcal{M}^{\ast\ast}$.
\end{proof}
\vspace{0.3ex}
\begin{remark}
In Theorem \ref{B_TH2}, if there exists an element $C$ in $\mathcal{Z}(\mathcal{A})$ and $C$ is a left separating point of $\mathcal{M}$, then $\delta$ satisfies (D2) if and only if $\delta$ is a generalized left derivation.
\end{remark}

\begin{theorem}\label{B_TH3}
Let $\mathcal{A}$ be a Banach algebra with the property ($\mathbb{B}$) and have a bounded approximate identity, $\mathcal{M}$ be a Banach left $\mathcal{A}$-module, and $\delta:\mathcal{A}\rightarrow\mathcal{M}$ be a continuous linear operator. Then $\delta$ satisfies (D3) if and only if $\delta$ is a generalized Jordan left derivation. \\
\indent Furthermore, if $\delta$ is a generalized Jordan left derivation with $\xi$ in $\mathcal{M}^{\ast\ast}$, then $\xi$ can be chosen in $\mathcal{M}$ in each of the following cases:
\vspace{-1ex}
\begin{enumerate}[\indent $\mathrm{(}i\mathrm{)}$]
  \item $\mathcal{A}$ is unital,
  \item $\mathcal{M}$ is a dual left $\mathcal{A}$-module.
\end{enumerate}
\end{theorem}

\begin{proof}[\bf Proof]
If $\delta$ satisfies (D3), then for each $A',B'$ in $\mathcal{A}$ with $A'B'=0$, we define the bilinear map $\phi~:~\mathcal{A}\times\mathcal{A}\rightarrow\mathcal{M}$ by $$\phi(A,B)=B'A\delta(BA')+BA'\delta(B'A)$$
for each $A,B$ in $\mathcal{A}$. For each $A,B$ in $\mathcal{A}$ with $AB=0$, since $B'ABA'=BA'B'A=0$, we have
$$\phi(A,B)=0.$$
Since $\mathcal{A}$ has the property ($\mathbb{B}$), we have that for each $A,B,C$ in $\mathcal{A}$, $\phi(AB,C)=\phi(A,BC)$, i.e.,
\begin{equation}\label{B_TH3_1}
B'AB\delta(CA')+CA'\delta(B'AB)=B'A\delta(BCA')+BCA'\delta(B'A).
\end{equation}
Now we fix $A,B,C$ in $\mathcal{A}$, and consider the bilinear mapping $\psi~:~\mathcal{A}\times\mathcal{A}\rightarrow\mathcal{M}$ such that for each $A',B'$ in $\mathcal{A}$,
$$\psi(A',B')=B'AB\delta(CA')+CA'\delta(B'AB)-B'A\delta(BCA')-BCA'\delta(B'A).$$
For each $A',B'$ in $\mathcal{A}$ with $A'B'=0$, on account of (\ref{B_TH3_1}), $\psi(A',B')=0$. Since $\mathcal{A}$ has the property ($\mathbb{B}$), it follows that for each $A',B',C'$ in $\mathcal{A}$, $$\psi(A'B',C')=\psi(A',B'C').$$
Thus, for each $A',B',C',A,B,C$ in $\mathcal{A}$,
\begin{multline}\label{B_TH3_2}
  C'AB\delta(CA'B')+CA'B'\delta(C'AB)-C'A\delta(BCA'B')-BCA'B'\delta(C'A)\\
  =B'C'AB\delta(CA')+CA'\delta(B'C'AB)-B'C'A\delta(BCA')-BCA'\delta(B'C'A).
\end{multline}
By taking into account that all the terms in (\ref{B_TH3_2}) involve $C'A$ and $CA'$, we concluded that for each $A,B,C,D$ in $\mathcal{A}$,
\begin{multline}\label{B_TH3_3}
  AB\delta(CD)+CD\delta(AB)-A\delta(BCD)-BCD\delta(A)\\
  +DA\delta(BC)+BC\delta(DA)-DAB\delta(C)-C\delta(DAB)=0.
\end{multline}
Let $(\rho_i)_{i\in \mathcal{I}}$ be a bounded approximate identity of $\mathcal{A}$. Since the net $(\delta(\rho_i))_{i\in \mathcal{I}}$ is bounded, it has a convergent subnet. Without loss of generality, we can assume that there exists an element $\xi$ in $\mathcal{M}^{\ast\ast}$ such that $\lim\limits_{i\in \mathcal{I}}\delta(\rho_i)=\xi$ with respect to the topology $\sigma(\mathcal{M}^{\ast\ast},\mathcal{M}^{\ast})$.
Applying (\ref{B_TH3_3}) with $A=C=\rho_i$, we have
\begin{multline*}
  \rho_i B\delta(\rho_i D)+\rho_i D\delta(\rho_i B)-\rho_i \delta(B\rho_i D)-B\rho_i D\delta(\rho_i )\\
  +D\rho_i \delta(B\rho_i )+B\rho_i \delta(D\rho_i )-D\rho_i B\delta(\rho_i )-\rho_i \delta(D\rho_i B)=0.
\end{multline*}
Then with respect to the topology $\sigma(\mathcal{M}^{\ast\ast},\mathcal{M}^{\ast})$,
\begin{align*}
 \delta(B\circ D)~=&~\delta(BD)+\delta(DB) \\
  =&~\lim\limits_{i\in \mathcal{I}}~(\rho_i \delta(B\rho_i D)+\rho_i \delta(D\rho_i B)) \\
  =&~\lim\limits_{i\in \mathcal{I}}~(\rho_i B\delta(\rho_i D)+\rho_i D\delta(\rho_i B)-B\rho_i D\delta(\rho_i )\\
  ~&~+D\rho_i  \delta(B\rho_i )+B\rho_i \delta(D\rho_i )-D\rho_i B\delta(\rho_i )) \\
  =&~ B\delta(D)+ D\delta(B)-BD\xi+D\delta(B)+B\delta(D)-DB\xi\\
  =&~ 2B\delta(D)+ 2D\delta(B)-(B\circ D)\xi.
\end{align*}
Thus there exists $\xi$ in $\mathcal{M}^{\ast\ast}$ such that for each $B,D$ in $\mathcal{A}$,
$$\delta(B\circ D)=2B\delta(D)+2D\delta(B)-(B\circ D)\xi,$$
i.e., $\delta$ is a generalized Jordan left derivation.

If $\delta$ is a generalized Jordan left derivation, it's obvious that for each $A,B$ in $\mathcal{A}$ with $AB=BA=0$, we have that
\begin{multline*}
  0=\delta(AB)+\delta(BA)=\delta(A\circ B)\\
  =2A\delta(B)+2B\delta(A)-(A\circ B)\xi=2(A\delta(B)+B\delta(A)).
\end{multline*}
Thus, $A\delta(B)+B\delta(A)=0$ for each $A,B$ in $\mathcal{A}$ with $AB=BA=0$.

Furthermore, if $\delta$ is a generalized Jordan left derivation with $\xi$ in $\mathcal{M}^{\ast\ast}$, then by Propositions \ref{B_prop3} and \ref{B_prop4}, $\xi$ can be chosen in $\mathcal{M}$ if $\mathcal{A}$ is unital, or if $\mathcal{M}$ is a dual left $\mathcal{A}$-module.
\end{proof}

\begin{theorem}\label{B_TH4}
Let $\mathcal{A}$ and $\mathcal{M}$ be as in Theorem \ref{B_TH3},
and $\delta:\mathcal{A}\rightarrow\mathcal{M}$ be a continuous linear operator. Then $\delta$ satisfies (D4) if and only if $\delta$ is a generalized Jordan left derivation. \\
\indent Furthermore, if $\delta$ is a generalized Jordan left derivation with $\xi$ in $\mathcal{M}^{\ast\ast}$, then $\xi$ can be chosen in $\mathcal{M}$ in each of the following cases:
\vspace{-1ex}
\begin{enumerate}[\indent $\mathrm{(}i\mathrm{)}$]
  \item $\mathcal{A}$ is unital,
  \item $\mathcal{M}$ is a dual left $\mathcal{A}$-module.
\end{enumerate}
\end{theorem}

\begin{proof}[\bf Proof]
If $\delta$ satisfies (D4), we consider $A,B$ in $\mathcal{A}$ with $AB=BA=0$, then $A\circ B=0$. With property (D4), we have that $A\delta(B)+B\delta(A)=0$, and $\delta$ satisfies (D3). By Theorem \ref{B_TH3}, $\delta$ is a generalized Jordan left derivation.

And if $\delta$ is a generalized Jordan left derivation, it's obvious that for each $A,B$ in $\mathcal{A}$ with $A\circ B=0$, we have
$$0=\delta(A\circ B)=2A\delta(B)+2B\delta(A)-(A\circ B)\xi=2(A\delta(B)+B\delta(A)).$$
Thus, $A\delta(B)+B\delta(A)=0$ for each $A,B$ in $\mathcal{A}$ with $A\circ B=0$.
\end{proof}

\begin{corollary}\label{B_CO1}
Let $\mathcal{A}$ be a Banach algebra with the property (B) and have a bounded approximate identity, $\mathcal{M}$ be a Banach left $\mathcal{A}$-module, and $\delta:\mathcal{A}\rightarrow\mathcal{M}$ be a continuous linear operator. Then the following conditions are equivalent:
\vspace{-1ex}
\begin{enumerate}[\indent $\mathrm{(}i\mathrm{)}$]
  \item $\delta$ is a generalized Jordan left derivation,
  \item $\delta$ satisfies (D3),
  \item $\delta$ satisfies (D4).
\end{enumerate}
Furthermore, with one of the above conditions, we have that
\begin{center}
$\delta$ is a Jordan left derivation~~$\Leftrightarrow$~~$\lim\limits_{i\in \mathcal{I}}~A^2\cdot\delta(\rho_i)=0$ for each $A$ in $\mathcal{A}$.
\end{center}
However, if $\mathcal{A}$ has identity $I$, then
\begin{center}
$\delta$ is a Jordan left derivation~~$\Leftrightarrow$~~$I\cdot\delta(I)=0$.
\end{center}
\end{corollary}

\begin{corollary}\label{B_CO2}
Let $\mathcal{A}$ be a Banach algebra with the property ($\mathbb{B}$) and have a bounded approximate identity, $\mathcal{M}$ be a Banach left $\mathcal{A}$-module with $Rann_{\mathcal{M}}(\mathcal{A})=\{0\}$, and $\delta:\mathcal{A}\rightarrow\mathcal{M}$ be a continuous linear operator. Then when $\delta$ satisfies (D1) or (D2),
we have that for each $C$ in $\mathcal{A}$, $\tilde{\delta}=C\cdot\delta$ is a generalized left derivation.\\
\indent Furthermore, if $\mathcal{A}$ has identity $I$, then
\begin{center}
$\tilde{\delta}$ is a left derivation~~$\Leftrightarrow$~~$I\cdot\tilde{\delta}(I)=0$.
\end{center}
\indent Or if $\mathcal{M}$ is a dual left $\mathcal{A}$-module, and $\mathcal{M}_{\ast}$ is a predual right $\mathcal{A}$-module of $\mathcal{M}$, then
\begin{align*}
 (i)~\tilde{\delta}~is~a~left~ derivation~&\Leftrightarrow~(ii)~\lim\limits_{i\in \mathcal{I}}~C\cdot\tilde{\delta}(\rho_i)=0~for~each~C~in~\mathcal{A}\\
 ~&\Leftrightarrow~(iii)~\sigma(\mathcal{M},\mathcal{M}_{\ast})-\lim\limits_{i\in \mathcal{I}}~\tilde{\delta}(\rho_i)=0.
\end{align*}
\end{corollary}

\end{document}